\documentclass[reqno,12pt]{amsart}
\usepackage{amssymb}
\usepackage{mathrsfs}
\usepackage{txfonts}
\usepackage{amsfonts}
\usepackage{amstext}
\usepackage{amssymb}
\usepackage{amsmath}
\usepackage{graphicx}
\usepackage{hyperref}
\usepackage{url}
\usepackage{amssymb}
\usepackage{subfig}

\hypersetup{
        colorlinks   = true,
        citecolor    = blue,
        linkcolor    = blue,
        urlcolor     = blue
}
\allowdisplaybreaks
\newtheorem{theorem}{Theorem}[section]

\newtheorem{lemma}[theorem]{Lemma}
\newtheorem{definition}[theorem]{Definition}
\newtheorem{corollary}[theorem]{Corollary}
\newtheorem{remark}[theorem]{Remark}
\renewcommand{\theequation}{\thesection.\arabic{equation}}

\newenvironment{acknowledgement}{\smallskip{\sc Acknowledgement.}\rm}{\smallskip}

\numberwithin{equation}{section}
\newcounter{counterConstant}

\input{tcilatex}

\def\Xint#1{\mathchoice
{\XXint\displaystyle\textstyle{#1}}%
{\XXint\textstyle\scriptstyle{#1}}%
{\XXint\scriptstyle\scriptscriptstyle{#1}}%
{\XXint\scriptscriptstyle\scriptscriptstyle{#1}}%
\!\int}
\def\XXint#1#2#3{{\setbox0=\hbox{$#1{#2#3}{\int}$ }
\vcenter{\hbox{$#2#3$ }}\kern-.6\wd0}}

\def\oint{\Xint-}

\def\FRAME#1#2#3#4#5#6#7#8
{
 \begin{figure}[ptbh]
 \begin{center}
 \includegraphics[height=#3]{#7}
 \caption{#5}
 \label{#6}
 \end{center}
 \end{figure}
}

\def\enddoc{\end{document}}

\begin{document}
\title[]{Algebraic dependence number and cardinality of generating iterated function systems}
\author[Zhang]{Junda Zhang}
\address{School of Mathematics, South China University of Technology,
Guangzhou 510641, China.}
\email{summerfish@scut.edu.cn}

\begin{abstract}
For a dust-like self-similar set (generated by IFSs with the strong separation condition), Elekes, Keleti and M\'{a}th\'{e} found an invariant, called `algebraic dependence number', by considering its generating IFSs and isometry invariant self-similar measures.
We find an intrinsic quantitative characterisation of this number: it is the dimension over $\mathbb{Q}$ of the vector space generated by the logarithms of all the common ratios of infinite geometric sequences in the gap length set, minus 1.
With this concept, we present a lower bound on the cardinality of generating IFS (with or without separation conditions) in terms of the gap lengths of a dust-like set.
We also establish analogous result for dust-like graph-directed attractors on complete metric spaces. This is a new application of the ratio analysis method and the gap sequence.
\end{abstract}

\date{\today}
\maketitle
\tableofcontents

\section{Introduction}
We recall some basic concepts in fractal geometry. In this paper, a finite set of
distinct contracting similarities $\{S_{i}\}_{i=1}^{m}$ on a complete metric space $(M,d)$ is called an \emph{iterated function system} (IFS), where for each $i,\ d(S_i(x),S_i(y))=r_id(x,y)\ \ (0<r_i<1$ is the {\em contraction ratio} of $S_i$). According to \cite{Hutchison1981}, the \emph{attractor} of
the IFS is the unique nonempty compact set $K\subset M$ such
that
\begin{equation}
K=\mathop{\textstyle \bigcup }\limits_{i=1}^{m}S_{i}(K),  \label{attract}
\end{equation}%
which is called a \textit{self-similar set}. We call $\{S_{i}\}_{i=1}^{m}$ a {\em generating IFS} of $K$.
We say that the \emph{strong separation condition} (SSC) holds for the IFS when the union is disjoint in \eqref{attract}, and such $K$ is called \textit{dust -like}.

This paper is motivated by the inverse fractal problem in \cite{DengLau2013,DengLau2017,FengWang2009,LiYao2017}. Given a dust-like set, what can be said about its generating IFSs? This problem is also related to tiling theory and image compression, see \cite[Section 1]{FengWang2009} for details.
Our results reveal the connection between the gap lengths of a dust-like set and its {\em algebraic dependence number} (Theorem \ref{yzx}), which can be used to give a lower bound on the cardinality of the generating IFS (Corollary \ref{zjd}). The definition of gap lengths is intrinsic, and it is simply the lengths of the disjoint (non-empty bounded) complementary intervals for a dust-like set on $\mathbb{R}$.

In \cite{Elekes2010}, the {\em algebraic dependence number} of an IFS is defined as the dimension over $\mathbb{Q}$ of the vector space generated by the logarithms of the contraction ratios of all similarities, minus 1. For a dust-like self-similar set in $\mathbb{R}^{n}$, \cite[Theorem 5.7]{Elekes2010} shows that all its generating IFSs with the SSC have the same algebraic dependence number, so this number can be defined for a dust-like self-similar set as its intrinsic characteristic (independent of generating IFSs). This is done by considering certain measures on the dust-like self-similar set, while our characterisation is direct and intrinsic (without considering generating IFSs and measures), and we generalise this to compact metric spaces. To be more detailed, our Theorem \ref{yzx} gives an intrinsic quantitative characterisation of this number: it is the dimension over $\mathbb{Q}$ of the vector space generated by the logarithms of all the common ratios of infinite geometric sequences in the gap length set (the collection of all gap lengths), minus 1.

We also consider graph-directed IFSs and graph-directed attractors based on a given directed graph, which generalizes the concept of IFS and self-similar sets. It is introduced in \cite{MauldinWil1988} and has been studied intensively, including the geometric aspects (see for example \cite{DasEdg2005,DasNgai2004, EdgarMau1992,Farkas2019,GMR2016,NiWen2014,Sascha2017}) and fractal analysis \cite{Bodin2007,CQTY2022,HamblyNyb2003,Metz2004}.
A directed graph $\left( V,E\right)$ consists of a finite set of vertices $V$ and a finite set of directed edges $E$ with loops and multiple edges allowed. Given a directed graph, we associate each vertex $v\in V$ with a metric space $(M^v,d^v)$. Let $E_{uv}\subset E$ be
the set of edges from vertex $u$ to $v$. A graph-directed iterated function system (GD-IFS) $\left\{ S_{e}:e\in E_{uv}\right\} $ consists of contracting similarities $S_{e}$ from $M_{v}$ to $M_u$, that is,
\begin{equation}d^u(S_e(x),S_e(y))=r_ed^v(x,y)\ \ (0<r_e<1)\label{rere}\end{equation} for all $x,y\in M^v$, where $r_e$ is the contraction ratio of $S_{e}$. We will assume $d_{u}\geq 2$ for all $u\in V$ to avoid singletons, where $d_{u}$ is
the number of directed edges leaving $u$ (see related discussion in \cite[pp.607]{EdgarMau1992}).
For
a GD-IFS $\left( V,E,\left( S_{e}\right) _{e\in E}\right) $ based on such a
directed graph, there exists a unique list of non-empty compact sets $(F_{u}\subset
M^u)_{u\in V}$ such that, for all $u\in V,$
\begin{equation}
F_{u}=\mathop{\textstyle \bigcup }\limits_{v\in V}\mathop{\textstyle \bigcup
}\limits_{e\in E_{uv}}S_{e}(F_{v}),  \label{gdattract}
\end{equation}%
see \cite[Theorem 4.3.5 on p.128]{Edgar}. We call
the above $\left( F_{u}\right) _{u\in V}$ the \emph{(list of) attractors }of
the GD-IFS, and each $F_{u}$ is called a \emph{GD-attractor}. We say that the \emph{strong separation condition} (SSC) holds for the GD-IFS, if the union is disjoint in \eqref{gdattract} for each $u\in V$, and the GD-attractors are called \textit{dust-like}. Clearly, an IFS and its attractor can be viewed as a GD-IFS and its GD-attractor when $\#V=1$, where `$\#$' denotes the cardinality of a finite set throughout the paper.

The graph-directed setting appear frequently in dynamical systems and the study of self-similar sets. Certain complex dynamical systems can be regarded as
conformal GD-IFSs using a Markov partition \cite[Section 5.5]{FalconerTechniques}. The orthogonal projection of
certain self-similar sets may be GD-attractors \cite[Theorem 1.1]{Farkas2016}.
When exact overlap occurs in an IFS, one can use a corresponding graph-directed system to study its attractor (see for example \cite{DengXi2016,HLR2003}). Also in fractal analysis, GD-IFSs are used to determine whether the two walk-dimensions coincide for p.c.f. self-similar sets \cite{GuLau.2020.TAMS}.
 They are also related to tiling problems \cite{RWY2014} and Automata \cite{CLR2015}.

The idea is applying the ratio analysis method in \cite{FHZJFG} to the gap length set of a dust-like GD-attractor, which exhibits a new application of the ratio analysis method and the gap sequence (see \cite{RRY2008} for applications to the Lipschitz equivalence problem and box dimension estimate).
Our results are useful for \textit{inhomogeneous} (GD-)IFSs and (GD-)attractors, where the contraction ratios are not all equal.
Separation conditions are often required in fractal geometry to obtain some precise structure information, and relaxing them can be very difficult in many problems. In our paper, the SSC is (only) used for a precise formula of the gap length set \cite[Theorem 2]{DengXi2016} for ratio analysis.
We remark that the SSC is a common assumption for inhomogeneous IFSs in related problems
 (see for example \cite{Algom2018,Elekes2010,FengXiong2018}). In Corollary \ref{zyd} we use the `full-measure' condition as in \cite[Theorem 4.1]{FengWang2009} to remove the SSC condition.

This paper is organised as follows. In Section 2, we introduce the definition and expression of the gap length set of a dust-like GD-attractor, and the required ratio analysis lemmas in \cite{FHZJFG}. In Section 3, we present our results. In Section 3.1, we use ratio analysis on the gap length set to obtain some lemmas and the logarithmic commensurability for SSC generating IFSs (Theorem \ref{blah}), which gives an alternative proof of \cite[Theorem 5.7]{Elekes2010}. In Section 3.2, we present Theorem \ref{yzx}, which characterizes the algebraic dependence number of (GD-)IFS in terms of (the geometric sequences in) the gap length set. In Section 3.3, we present the lower bound estimate on the cardinality of generating IFS, with or without separation conditions (Corollaries \ref{zjd} and \ref{zyd}).

\section{Preliminaries}
\subsection{Gap lengths}
The concept and lemma in this subsection are from \cite{DengXi2016}, which concerns the gap sequence, a natural way to define the size and number of holes for disconnected sets (see \cite{RRY2008} for details on $\mathbb{R}^n$).
For a compact metric space $(K,d)$ and $\delta>0$, denote its diameter by {\rm diam}$(K)$. We say that $x,y\in K$ are $\delta$-equivalent, if there is a $\delta$-chain connecting $x$ and $y$, that is, there are points $x_i\in K\ (i=1,2,\cdots,t)$ such that $x_1=x,x_t=y,d(x_i,x_{i+1})\leq \delta$ for all $i$.
According to this equivalence, we can divide $K$ into several $\delta$-equivalence
classes (intuitively, these small pieces are separated by a distance $\delta$, while any two points in each piece can be connected by some $\delta$-chain).
Let $\kappa(\delta)$ be the number of $\delta$-equivalence
classes of $K$, which is finite due to the compactness of $K$ (since two points in an open $\delta/2$-metric ball are automatically $\delta$-equivalent, $\kappa(\delta)$ does not exceed the covering number of $K$ using open $\delta/2$-metric balls).

\begin{definition}[Gap length set]
The discontinuous points of the
function $\kappa(\delta) (0<\delta<${\rm diam}$(K)$) are called the {\em gap lengths} of $K$, and the collection of all the gap lengths of $K$ is called the {\em gap length set} of $K$, denoted by $\mathrm{GL}(K)$ (that is, the set of discontinuous points of $\kappa$).
\end{definition}
\begin{remark}
For compact sets on $\mathbb{R}$, this definition is intuitive and coincides with \cite[Definition 2.1]{FHZJFG}, see explanation and an example (middle-third Cantor set) in \cite[Section 1]{RRY2008}.

Note that the function $\kappa(\delta)$
is non-increasing in $\delta\in (0,${\rm diam}$(K))$, since a $\delta$-equivalence
class must be a $\delta'$-equivalence
class when $\delta<\delta'$. Thus the gap lengths are at-most countably many, and arranging them with corresponding multiplicity (according to the value of $\kappa$) gives the {\em gap sequence} in \cite[Definition 2]{DengXi2016}, but we do not concern such multiplicity in our paper.
\end{remark}

\begin{definition}[Contraction ratio set]
The \emph{contraction ratio set} of a GD-IFS is defined to be the
set of the contraction ratios of the similarities,
that is, $\{r_e:e\in E\}$ using the notation in \eqref{rere}.
\end{definition}
 The following Lemma is a direct corollary of \cite[Theorem 2]{DengXi2016}. We define the product of $A,\ B\subset \mathbb{R}$ to be $AB=\{ab:\ a\in A,\ b\in B\}.$  We regard a real number as a set when encountering the product with a set in $\mathbb {R}$, and $AB$ is empty if $A$ is empty. Given a GD-IFS and a directed path $\mathbf{e}=e_{1}e_{2}\cdots
e_{k}$ with edges $e_{i}$ (for which
the terminal vertex of $e_{i}$ is the initial vertex of $e_{i+1}$ when $i=1,%
\mathbb{\cdots },k-1$), define $r_{\mathbf{e}}:=r_{e_{1}}r_{e_{2}}\cdots r_{e_{k}}$.
\begin{lemma}
Let $G=(V,E)$ be a directed graph satisfying $d_w\geq 2$ for all $w\in V$, and let $(M^u)_{u\in V}$ be a list of complete metric spaces. Suppose that for some $u\in V$, $F_{u}\subset M^u$ is the GD-attractor of some GD-IFS (based on $G$) with the SSC.
Then for each $%
u\in V$, there exist finite sets $\Lambda_u,\ \Gamma_u\subset
	(0,\infty)$ such that
\begin{equation}\label{GDRSSC}
\mathrm{GL}(F_u)=\Gamma_u\bigcup\Lambda_u\bigcup\left(\bigcup_{v\in V}\Lambda_v R_{uv}\right)\text{\ where\ }R_{uv}:=\big\{r_{\mathbf{e}}:\mathbf{e}
\ \text{is a directed path from $u$ to $v$}\big\}.
\end{equation}
\end{lemma}
\begin{remark}When there is no directed path from $u$ to $v$, then $R_{uv}=\emptyset$. We mention that \cite[Theorem 2]{DengXi2016} gives the construction of $\Lambda_u$ and $\Gamma_u$, although it is not used in our paper. Given a GD-IFS and its attractor list $(F_u\subset M^u)_{u\in V}$, we obtain $\mathrm{GL}(F_u)$ and define $\Gamma_u=\mathrm{GL}(F_u)\bigcap [\delta,\infty)$ for all $u\in V$, where
$\delta:=\inf\Big\{d^u(x_e,y_e):\ x_e\in S_e(F_v),\ y_e\in
F_u\setminus S_e(F_v),\ e\in E_{uv},\ u,\ v\in V\Big\}>0$ due to the SSC. Then define $\Lambda_u=\bigcup_{u\in V}\bigcup_{e\in E_{uv}}r_e(\Gamma_v\bigcap[\delta,r_e^{-1}\delta))$ for each $u\in V$, from which we know $\Lambda_u\subset (0,\delta)$.\end{remark}
We will use the following notions in \cite[Definition 2.5]{FHZJFG} and a corollary of the above lemma.
\begin{definition}
For a finite set $A=\{a_{i}\}_{i=1}^{n}\subset (0,\infty )$, define $A^{%
\mathbb{Z}_{+}^{\ast }}$ (resp. $A^{\mathbb{Q}_{+}^{\ast }}$) to be the union of all products $\mathop{\textstyle \prod }%
\limits_{i=1}^{n}a_{i}^{m_{i}}$ where $(m_{i})_{i=1}^{n}$ are non-zero
vectors whose entries are nonnegative integers (resp. nonnegative rationals). Let $A^{\mathbb{Z}_{+}}=\{1\}\cup A^{\mathbb{Z}_{+}^{\ast }}$.\label{aqz}
\end{definition}

\begin{corollary}\label{SSRSSC}Let $M$ be a complete metric space, and $K\subset M$ be the attractor of an IFS satisfying the SSC with contraction ratio set $X$. Then there exist two finite sets $\Lambda,\ \Gamma\subset 	(0,\infty)$ such that
	\begin{equation}
		\mathrm{GL}(K)=\Gamma\bigcup\Lambda X^{\mathbb{Z}_+}.
\end{equation}\end{corollary}

\subsection{Ratio analysis}
In \cite{FHZJFG}, the ratio analysis method is used to analyse sets $\Theta $ of positive real
numbers, in terms of strictly decreasing geometric
sequences $\{\theta ^{\prime }r^{k}\}_{k=0}^{\infty }$ that are contained in
$\Theta $. Whenever we mention `geometric
sequence' in this paper, we always assume it is infinitely-long.
\begin{definition}(\cite[Definition 2.4]{FHZJFG})
Let $\Theta \subset (0,\infty )$. For $\theta \in \Theta $, denote by
\begin{equation}
R_{\Theta }(\theta )=\{r\in (0,1):\text{ there exists some }\theta ^{\prime
}\in \Theta \text{ such that }\theta \in \{\theta ^{\prime
}r^{k}\}_{k=0}^{\infty }\subset \Theta \},  \label{R}
\end{equation}%
the set of common ratios of strictly decreasing geometric sequences in $%
\Theta $ that contains $\theta $ (which may be
empty).
\end{definition}

This concept is quite natural since the gap length set of a dust-like attractor
contains many geometric sequences. By definition, there is an obvious monotoniticy: when
$\theta \in\Theta_1\subset\Theta_2\subset (0,\infty )$,\begin{equation}R_{\Theta_1 }(\theta )\subset R_{\Theta_2 }(\theta ),\label{dand}\end{equation}
since a geometric sequence in $\Theta_1$ is also in $\Theta_2$.

We will use the following two lemmas.

\begin{lemma}(\cite[Lemma 2.6 (i)]{FHZJFG})
\label{GDR}Let $A=\{a_{i}\}_{i=1}^{n}\subset (0,1)$ be a finite set, and $\lambda _{j}\ (j=1,\cdots ,m)$ be
positive real numbers (not necessarily distinct). Let $\Theta =%
\mathop{\textstyle \bigcup }\limits_{j=1}^{m}\lambda _{j}A_{j}$ where $%
A_{j}\subset A^{\mathbb{Z}_{+}}$ for $1\leq j\leq m$.
Then $R_{\Theta }(\theta )\subset A^{\mathbb{Q}_{+}^{\ast }}$ for all $%
\theta \in \Theta $.
\end{lemma}

\begin{lemma}(\cite[Corollary 2.7]{FHZJFG})
\label{SSR} Let $X\subset (0,1)$ and $\Lambda \subset (0,\infty )$ be two
finite sets. Then
\begin{equation*}
X^{\mathbb{Z}_{+}^{\ast }}\subset R_{\Lambda X^{\mathbb{Z}_{+}}}(\theta
)\subset X^{\mathbb{Q}_{+}^{\ast }}
\end{equation*}%
for every $\theta \in \Lambda X^{\mathbb{Z}_{+}}$.
\end{lemma}

\section{Our Results}
\subsection{Using ratio analysis to obtain logarithmic commensurability}
Logarithmic commensurability is central to the affine-embedding problem \cite{Algom2018, AmirHoc2018, Elekes2010, FHR2014, FengXiong2018, Pablo2019, Wu2019} and inverse fractal problem (generating IFS), which is stated as follows.
\begin{conjecture}\cite[Conjecture 1.2]{FHR2014}
Suppose that $K,\ F$ are two totally disconnected non-singleton self-similar sets in $\mathbb{R}^{n}$, which are the attractors of two IFSs with contraction ratio sets $X,\ Y$ respectively. Suppose that there exists an affine map $f$ in $\mathbb{R}^{n}$ such that $f(F)\subset K,$ then the {\em logarithmic commensurability} holds: $Y\subset X^{\mathbb{Q}_+^*}.$
\end{conjecture}
The following theorem can be deduced from \cite[Lemma 4.8]{Elekes2010} in $\mathbb{R}^n$, which uses geometric measure theoretic arguments. Here we present another proof without using measures.

\begin{theorem}\label{blah}
Let $M$ be a complete metric space. Suppose that $K\subset M$ has two generating IFSs satisfying the SSC with contraction ratio sets $A$ and $X$, respectively,
then $X\subset A^{\mathbb{Q}_+^*}, \ A\subset X^{\mathbb{Q}_+^*}. $
\end{theorem}
This theorem is a direct consequence of Lemma \ref{sscom}. To prove Lemma \ref{sscom}, we first prove a more general lemma for GD-attractors, which will be used later.
\begin{lemma}\label{gdcom}
Let $G=(V,E)$ be a directed graph satisfying $d_w\geq 2$ for all $w\in V$, and let $(M^u)_{u\in V}$ be a list of complete metric spaces. Suppose that for some $u\in V$, $F_{u}\subset M^u$ is the GD-attractor of some GD-IFS (based on $G$) satisfying the SSC with contraction ratio set $A$.
Then $R_{\mathrm{GL}(F_{u})}(\theta )\ \subset \
A^{\mathbb{Q}_{+}^{\ast }}$ for all $\theta \in \mathrm{GL}(F_{u})$.\end{lemma}
\begin{proof}
We apply Lemma \ref{GDR} to $\Theta=\mathrm{GL}(F_{u})$ given by (\ref{GDRSSC}), by regarding the numbers in $\Gamma_u$, $\Lambda_u$ and $\Lambda_v\ (v\in V)$ as $\lambda_j$, where the corresponding $A_j$ are $\{1\}$ and $R_{uv}$ respectively.
Since $A$ is the contraction ratio set, $r_{\mathbf{e}}\in A^{\mathbb{Z}_{+}},$ the conditions in Lemma \ref{GDR} are fulfilled, showing the desired.\end{proof}
\begin{lemma}\label{sscom}
Let $M$ be a complete metric space and $K\subset M$ be the attractor of an IFS satisfying the SSC with contraction ratio set $X$. Then for all small enough $\theta \in \mathrm{GL}(K)$,\begin{equation}
X\subset X^{\mathbb{Z}_{+}^{\ast }}\subset R_{\mathrm{GL}(K)}(\theta)\subset X^{\mathbb{Q}_{+}^{\ast }}.\label{hty}
\end{equation}\end{lemma}
\begin{proof}
For the left-hand side in \eqref{hty},
by Corollary \ref{SSRSSC} and our assumption, there exist two finite sets $\Lambda,\ \Gamma\subset
	(0,\infty)$ such that
	\begin{equation}
		\mathrm{GL}(K)=\Gamma\bigcup\Lambda X^{\mathbb{Z}_+}.\label{dkz}
\end{equation}
Thus by the monotonicity (\ref{dand}) and Lemma \ref{SSR}, we know that for $\theta \in \Lambda X^{\mathbb{Z}_{+}}\subset \mathrm{GL}(K),$
$$X^{\mathbb{Z}_{+}^{\ast }}\subset R_{\Lambda X^{\mathbb{Z}_{+}}}(\theta
)\subset  R_{\mathrm{GL}(K)}(\theta ),$$
showing the desired for all small enough $\theta \in \mathrm{GL}(K)$ (smaller than the numbers in the finite set $\Gamma$ in (\ref{dkz})).

The right-hand side in \eqref{hty} directly follows from Lemma \ref{gdcom}. The proof is complete.
\end{proof}
\begin{proof}[Proof of Theorem \ref{blah}] By Lemma \ref{sscom}, for all small enough $\theta \in \mathrm{GL}(K)$,\begin{equation*}
A\subset R_{\mathrm{GL}(K)}(\theta)\subset A^{\mathbb{Q}_{+}^{\ast }},
X\subset R_{\mathrm{GL}(K)}(\theta)\subset X^{\mathbb{Q}_{+}^{\ast }}.
\end{equation*}Thus $$A\subset R_{\mathrm{GL}(K)}(\theta)\subset X^{\mathbb{Q}_{+}^{\ast }},X\subset R_{\mathrm{GL}(K)}(\theta)\subset A^{\mathbb{Q}_{+}^{\ast }},$$which ends the proof.\end{proof}

\subsection{Intrinsic characterisation of algebraic dependence number}

We first generalize the definition of algebraic dependence number to GD-IFSs. Denote by $\mathrm{span}B$ the vector space generated by $B\subset \mathbb{R}$ over $\mathbb{Q}$, $\log X:=\{\log x:x\in X\}$ for a set $X\subset (0,\infty)$.
\begin{definition}[Algebraic (in)dependence number of GD-IFS] Given a GD-IFS with contraction ratio set $A$, define its {\em algebraic independence number} as the dimension of the vector space $\mathrm{span}\log A$, and its {\em algebraic dependence number} as its algebraic independence number minus 1.\end{definition}
 \begin{theorem}Suppose that $\digamma$ is a GD-IFS satisfying the SSC with GD-attractors $(F_{u}\subset
M^u)_{u\in V}$ (on complete metric spaces), based on a directed graph $(V,E)$ with $d_w\geq 2$ for all $w\in V$, then its algebraic independence number is no less than the dimension of \begin{equation}\label{yzy}\mathrm{span}\Big\{\mathrm{log}\,r:\ r\in \bigcup_{u\in V}\bigcup_{\theta\in \mathrm{GL}(F_u)}R_{\mathrm{GL}(F_u)}(\theta)\Big\}.\end{equation} In particular, for a dust-like self-similar set $K$ (on a complete metric space), its algebraic dependence number plus 1 equals the dimension of \begin{equation}\label{zxd}
  \mathrm{span}\Big\{\mathrm{log}\,r:\ r\in \bigcup_{\theta\in \mathrm{GL}(K)}R_{\mathrm{GL}(K)}(\theta)\Big\}=\mathrm{span}\{\mathrm{log}\,r: r\in R_{\mathrm{GL}(K)}(\theta)\}
\end{equation} for any small enough $\theta\in \mathrm{GL}(K)$
(the vector space generated by the logarithms of all the common ratios of infinite geometric sequences in $\mathrm{GL}(K)$ over $\mathbb{Q}$).\label{yzx}\end{theorem}
\begin{remark}
In this theorem, ``infinite geometric sequences in $\mathrm{GL}(K)$" has the same effect as ``strictly decreasing geometric sequences in $\mathrm{GL}(K)$", since the set $\mathrm{GL}(K)$ has an upper bound {\rm diam}$(K)$, which means that the infinite geometric sequences in $\mathrm{GL}(K)$ must have common ratio no greater than 1 (but $\mathrm{log}1=0$ is useless to be a basis for vector spaces). \end{remark}
In fact we prove a stronger property \eqref{cblcd} for dust-like self-similar sets: the vector space generated by the logarithms of the contraction ratios of any SSC generating IFS (over $\mathbb{Q}$) is exactly the vector space generated by the logarithms of all the common ratios of geometric sequences in the gap length set (over $\mathbb{Q}$).

\begin{proof}
For the first inclusion, by using Lemma \ref{gdcom} and taking the logarithm,\begin{equation*}B:=\Big\{\mathrm{log}\,r:\ r\in \bigcup_{u\in V}\bigcup_{\theta\in \mathrm{GL}(F_u)}R_{\mathrm{GL}(F_u)}(\theta)\Big\}\subset \log A^{\mathbb{Q}_{+}^{\ast }},\end{equation*} where $A$ is the contraction ratio set of $\digamma$. It follows by Definition \ref{aqz} that\begin{equation*}\mathrm{span}B\subset \mathrm{span}\log A^{\mathbb{Q}_{+}^{\ast }}=\mathrm{span}\log A. \end{equation*} Then we know the dimension of $\mathrm{span}B$ is no greater than that of $\mathrm{span}\log A$, which is exactly the algebraic independence number of $\digamma$.

For the second inclusion, by using Lemma \ref{sscom} and taking the logarithm,\begin{equation*}\log X\subset B'\subset \log X^{\mathbb{Q}_{+}^{\ast }},\end{equation*}
where $X$ is the contraction ratio set of any SSC generating IFS of $K$, and $B'$ can be either $\Big\{\mathrm{log}\,r:\ r\in \bigcup_{\theta\in \mathrm{GL}(K)}R_{\mathrm{GL}(K)}(\theta)\Big\}$ or $\{\mathrm{log}\,r: r\in R_{\mathrm{GL}(K)}(\theta)\}$ for any small enough $\theta\in \mathrm{GL}(K)$. Thus \begin{equation*}\mathrm{span}\log X\subset \mathrm{span}B'\subset\mathrm{span}\log X^{\mathbb{Q}_{+}^{\ast }}. \end{equation*}
Note that by Definition \ref{aqz}, $\mathrm{span}\log X=\mathrm{span}\log X^{\mathbb{Q}_{+}^{\ast }}$,
thus\begin{equation}\mathrm{span}\log X=\mathrm{span}B'=\mathrm{span}\log X^{\mathbb{Q}_{+}^{\ast }}, \label{cblcd}\end{equation}showing \eqref{zxd}.
Since the algebraic dependence number of $K$ is the algebraic dependence number of any SSC generating IFS, that is, the dimension of $\mathrm{span}\log X$, the proof is finished.
\end{proof}

\subsection{Lower bound on the cardinality of generating IFS}
Algebraic independence number provides a natural lower bound for the cardinality of generating IFSs, so we are able to present such lower bound in terms of the fractal itself (using gap lengths).
\begin{corollary}
Let $G=(V,E)$ be a directed graph satisfying $d_w\geq 2$ for all $w\in V$, and let $(M^u)_{u\in V}$ be a list of complete metric spaces. Suppose that for some $u\in V$, $F_{u}\subset M^u$ is the GD-attractor of some GD-IFS $\digamma$ satisfying the SSC (based on $G$). Then the number of edges in $G$ is no less than the dimension of $\mathrm{span}\Big\{\mathrm{log}\,r:\ r\in \bigcup_{\theta\in \mathrm{GL}(F_u)}R_{\mathrm{GL}(F_u)}(\theta)\Big\}$.

In particular, for a dust-like self-similar set $K\subset M$ (where $M$ is a complete metric space), the cardinality of its generating IFS satisfying the SSC is no less than dimension of $$\mathrm{span}\Big\{\mathrm{log}\,r:\ r\in \bigcup_{\theta\in \mathrm{GL}(K)}R_{\mathrm{GL}(K)}(\theta)\Big\}=\mathrm{span}\{\mathrm{log}\,r: r\in R_{\mathrm{GL}(K)}(\theta)\}$$ for any small enough $\theta\in \mathrm{GL}(K)$.\label{zjd}\end{corollary}
\begin{proof}
  For the first claim, just note that $\#E$, the number of edges in $G$, which is also the cardinality of contracting similarities, is no less than the cardinality of the contraction ratio set of $\digamma$, and thus no less than the algebraic independence number of $\digamma$. The first claim then follows by Theorem \ref{yzx} and that $\mathrm{span}\Big\{\mathrm{log}\,r:\ r\in \bigcup_{\theta\in \mathrm{GL}(F_u)}R_{\mathrm {GL}(F_u)}(\theta)\Big\}\subset \mathrm{span}\Big\{\mathrm{log}\,r:\ r\in \bigcup_{u\in V}\bigcup_{\theta\in \mathrm{GL}(F_u)}R_{\mathrm{GL}(F_u)}(\theta) \Big\}$.

  The second claim immediately follows from the first claim and \eqref{zxd}.
\end{proof}
\begin{remark}We delete the union ``$\bigcup_{u\in V}$'' in \eqref{yzy}, since we only have the information of one attractor at one vertex $u$. In addition, we only give a lower bound on the number of edges, not of vertices.\end{remark}

Furthermore, if the dust-like self-similar set is on $\mathbb{R}$ and is full-measure, we can remove the separation condition (SSC) on the generating IFS.
We say that a nonempty compact set $K$ is {\em full-measure} if
$
\mathcal{H}^{\mathrm{dim_{H}}K}(K)=(\mathrm{diam}(K))^{\mathrm{dim_{H}}K}
$ (where $\mathrm{dim_{H}}$ is the Hausdorff dimension).

\begin{corollary}Suppose that a full-measure set $K\subset \mathbb{R}$ is the attractor of an IFS $\Phi $ satisfying the SSC, then the cardinality of every generating IFS of $K$ is no less than the dimension of $\mathrm{span}\Big\{\mathrm{log}\,r:\ r\in \bigcup_{\theta\in \mathrm{GL}(K)}R_{\mathrm{GL}(K)}(\theta)\Big\}$. If further, the logarithm of the contraction ratio of each contracting similarity in $\Phi$ is different from each other and linearly independent over $\mathbb{Q}$, then $\Phi$ has the minimal cardinality among all generating IFSs of $K$.\label{zyd}\end{corollary}
\begin{proof}For any generating IFS $\{S_{i}\}_{i=1}^{m}$ of $K$, there exists $I\in \{1,2,\cdots,m\}$ such that, the sub-IFS $\Psi:=\{S_{i}\}_{i\in I}$ also has attractor $K$ and satisfies the SSC, and \begin{equation}
                                               m\geq \#\Psi.\label{llm}
                                             \end{equation} Indeed, whenever $S_i(K)\bigcap S_j(K)\neq\emptyset$, by \cite[The claim in the proof of Theorem 4.1]{FengWang2009}, $S_i(K)\subset S_j(K)$ (or $S_i(K)\supset S_j(K)$), so one can remove $i$ (or $j$) from $\{1,2,\cdots,m\}$, and repeat this process until the SSC is satisfied. The first inclusion then follows by using Corollary \ref{zjd} for the lower bound of $\#\Psi $, which is also a lower bound for $m$ by \eqref{llm}.

For the second inclusion, by assumption, the algebraic independence number of $\Phi $ is exactly its cardinality $\#\Phi $, thus the algebraic independence number of $\Psi $ is also $\#\Phi $ by \cite[Theorem 5.7]{Elekes2010}. Since the cardinality of $\Psi $ is no less than the algebraic independence number of $\Psi $, which is $\#\Phi $, we obtain $ m\geq \#\Phi$ by \eqref{llm}, showing the desired.\end{proof}

The full-measure condition is easy to verify by the similarities of an IFS on $\mathbb{R}$, see \cite[Remark 4.1]{FengWang2009} for a practical sufficient and necessary condition and examples.
We end our paper by a brief overview on the full-measure condition. It is firstly discussed in Hausdorff's paper \cite{Hausdorff}, then extended by Marion (called `perfect isotopic') \cite{Marion1985,Marion1986}, and independently by Ayer and Strichartz \cite{AyerStr1999} on self-similar sets (several years later): they give the full-measure criterion for dust-like self-similar sets on $\mathbb{R}$. There is no known analogous criterion in higher dimensional case to the best of our knowledge. The study of full-measure condition for a GD-attractor is more complicated than that for self-similar set (see \cite[Theorem 4.6]{BooreFal2013}), and there is no known analogous criterion even on $\mathbb{R}$ to the best of our knowledge.

\begin{acknowledgement}
The author thanks Kenneth Falconer for valuable suggestions.
\end{acknowledgement}


\end{document}